\documentclass[letterpaper, 10 pt, conference]{ieeeconf}  

\IEEEoverridecommandlockouts                             
\usepackage{amsmath} 
\usepackage{amssymb}  

\usepackage{amsthm}
\usepackage{color}
\usepackage{xcolor}
\usepackage{mathrsfs}
\usepackage{mathtools}
\usepackage{algorithm}
\usepackage{algorithmic}
\mathtoolsset{showonlyrefs}
\usepackage{hyperref}
\usepackage{comment}
\usepackage{cancel,ulem,booktabs,multirow}

\usepackage{flushend}

\newcommand{\Prob}{\mathbb{P}}
\newcommand{\R}{\mathbb{R}}

\newcommand{\F}{\mathscr{F}}

\newcommand{\E}{\mathbb{E}}

\newcommand{\KL}{\text{D}_{\text{KL}}}
\newcommand{\ud}{\,\mathrm{d}}
\newcommand{\QV}[2]{\left\langle #1, #2\right\rangle}

\newcommand{\BAR}[1]{\overline{#1}}

\newcommand{\REG}{\text{reg}}
\newcommand{\red}{\text{red}}
\newcommand{\transpose}{^{\operatorname{T}}}

\newtheorem{theorem}{Theorem}
\newtheorem{lemma}{Lemma}
\newtheorem{prop}{Proposition}

\newtheorem{remark}{Remark}
\newtheorem{assu}{Assumption}
\newtheorem{cor}{Corollary}

\title{\LARGE \bf
Integrating Sequential Hypothesis Testing into Adversarial Games: A Sun Zi-Inspired Framework
}

\author{Haosheng Zhou, Daniel Ralston, Xu Yang and Ruimeng Hu
\thanks{This work was partially supported by the ONR grant under \#N00014-24-1-2432, the Simons Foundation (MP-TSM-00002783) and the NSF grant DMS-2109116, DMS-2420988.}
\thanks{H. Zhou is with the Department of Statistics and Applied Probability, University of California, Santa Barbara, CA 93106, USA
        {\tt\small hzhou593@ucsb.edu}}%
\thanks{D. Ralston is with the Department of Mathematics, University of California, Santa Barbara, CA 93106, USA
        {\tt\small danielralston@ucsb.edu}}%
\thanks{X. Yang is with the Department of Mathematics, University of California, Santa Barbara, CA 93106, USA
        {\tt\small xy6@ucsb.edu}}%
\thanks{R. Hu is with the Department of Mathematics, and Department of Statistics and Applied Probability, University of California, Santa Barbara, CA 93106, USA
        {\tt\small rhu@ucsb.edu}}%
}

\begin{document}

\maketitle
\thispagestyle{empty}
\pagestyle{empty}

\begin{abstract}

This paper investigates the interplay between sequential hypothesis testing (SHT) and adversarial decision-making in partially observable games, focusing on the deceptive strategies of red and blue teams. Inspired by Sun Zi's \textit{The Art of War} and its emphasis on deception, we develop a novel framework to both deceive adversaries and counter their deceptive tactics. We model this interaction as a Stackelberg game where the blue team, as the follower, optimizes its controls to achieve its goals while misleading the red team into forming incorrect beliefs on its intentions. The red team, as the leader, strategically constructs and instills false beliefs through the blue team's envisioned SHT to manipulate the blue team’s behavior and reveal its true objectives. The blue team’s optimization problem balances the fulfillment of its primary objectives and the level of misdirection, while the red team coaxes the blue team into behaving consistently with its actual intentions. We derive a semi-explicit solution for the blue team’s control problem within a linear-quadratic framework, and illustrate how the red team leverages leaked information from the blue team to counteract deception. Numerical experiments validate the model, showcasing the effectiveness of deception-driven strategies in adversarial systems. These findings integrate ancient strategic insights with modern control and game theory, providing a foundation for further exploration in adversarial decision-making, such as cybersecurity, autonomous systems, and financial markets.


\end{abstract}


\section{Introduction}

Deception has long been recognized as a critical component of strategic interactions, with its roots embedded in ancient military philosophy. Sun Zi's \textit{The Art of War} emphasizes the fundamental role of deception in strategic decision-making, stating, \textit{"All warfare is based on deception"}~\cite{zi2007art}.
In modern contexts, the concept of deception extends far beyond traditional battlefields, finding relevance in areas such as cybersecurity \cite{aggarwal2016cyber}, autonomous systems \cite{arkin2011moral}, and financial markets (\cite{gerschlager2005deception,back2000imperfect}).
In these domains, adversarial decision-making plays a crucial role, where opposing agents employ strategies to outmaneuver and mislead each other (e.g., \cite{yager2008knowledge}). 
Despite the prevalence of deception, mathematically modeling deceptive interactions is challenging.
Unlike well established optimization (collaborative games) and fixed point problems (competitive games) commonly seen in the control literature, deception is not associated with a certain type of mathematical problems. This provides extra flexibility, but also results in different modeling approaches from various perspectives.
In this paper, we model deceptive interactions through a red team-blue team setting (e.g., \cite{rajendran2011blue}), with a particular focus on the interplay between sequential hypothesis testing and stochastic control. 

In our framework, the blue team pursues its primary goals while anticipating the red team’s inference process and actively misleading it through strategic misdirection. In response, the red team counters by subtly planting false beliefs about the test into the blue team’s perception, carefully designed to avoid arousing suspicion, thereby coaxing it into unconsciously revealing its true objective. Inspired by Sun Zi’s principle, ``When we are near, we must make the enemy believe we are far away; when far away, we must make him believe we are near'' \cite{zi2007art}, we model the blue team’s dynamics in a velocity-position setting. 

To formalize inference, we adopt \textit{sequential hypothesis testing (SHT)}, a statistical method that evaluates hypotheses in real-time as data accumulates \cite{tartakovsky2014sequential,goodman2007adaptive,schonbrodt2017sequential} and is well-suited for adversarial settings due to its efficiency in real-time decision-making. We embed SHT into a stochastic control framework for the blue team and cast the red-blue interaction as a Stackelberg game, with the red team leading the strategy.

The contributions of this paper are threefold:

\noindent\textbf{1. Linear-Quadratic (LQ) Model for Strategic Misdirection}: We develop a model for the blue team’s deceptive strategies while accomplishing its primary task by incorporating SHT statistics into linear-quadratic control formulations. 
The proposed model handles partially observable settings -- common in adversarial systems -- without requiring filtering (\cite{bain2009fundamentals,davis1977linear}), and departs from passive robustness frameworks by enabling active deception (\cite{taskesen2024distributionally,hakobyan2024wasserstein,moon2016linear,bauso2016robust}). To our knowledge, such integration of test statistics into control design is novel.

\noindent\textbf{2. Stackelberg Game}: We formulate the adversarial interaction as a Stackelberg game, with the red team as leader and the blue team as follower. Anticipating that the blue team envisions its use of hypothesis testing and possible misdirection, the red team strategically embeds alternative beliefs to steer the blue team toward reduced deception. This alignment improves the red team’s ability to infer the blue team’s true intentions.


\noindent\textbf{3. Numerical Validation}: We derive a semi-explicit solution for the blue team and propose iterative and learning-based methods for the red team. Numerical results support the theory and demonstrate the benefits of deception-aware strategies under adversarial dynamics.

\smallskip
\noindent\textbf{Related Literature.} Adversarial interactions are often modeled as partially observable stochastic games (POSGs) (\cite{horak2019solving,ma2024sub,liu2022sample}), which effectively capture the uncertainty and incomplete information inherent in domains such as cybersecurity, where attackers and defenders rely on noisy signals, or autonomous systems, where agents operate with incomplete sensor data. 
However, POSGs typically assume zero-sum structures and rely on computationally intensive methods such as partially observable Markov decision processes (POMDPs) or belief-space planning in which agents maintain and update probability distributions over hidden states (\cite{kurniawati2009sarsop,roy2005finding,kim2019pomhdp}), with limited theoretical grounding. 
Antagonistic control frameworks (\cite{lipp2016antagonistic}) offer an alternative, modeling attacks on control systems via cost maximization, but are generally restricted to deterministic settings and requires convex constraints on the state-action space to ensure well-posedness.
Given these challenges, we take an alternative approach by employing SHT, which offers a more direct and computationally efficient framework for modeling strategic deception and inference in adversarial settings.

\section{The Linear-Quadratic Model}\label{sec:model}
In this section, we present the blue team's bi-objective optimization problem: fulfill a primary task while simultaneously executing strategic misdirection. Sec.~\ref{sec:IIA} formulates the primary task via linear-quadratic dynamics. Sec.~\ref{sec:IIB} introduces sequential hypothesis testing (SHT) and its related mathematical backgrounds, modeling strategic misdirection through its likelihood ratio statistic. Sec.~\ref{sec:IIC} integrates both objectives into a single linear-quadratic control framework, highlighting the trade-off between primary task performance and intent concealment. Using dynamic programming and an appropriate ansatz, we reduce the control problem to solving a system of ordinary differential equations (ODEs), derive a semi-explicit solution, and establish well-posedness, which offers key insights into the model.



\subsection{The Primary Task -- a Baseline Model}\label{sec:IIA}

Consider a filtered probability space \((\Omega,\F,\{\F_t\}_{t\geq 0},\mathbb{P})\) supporting independent Brownian motions \(\{B_t, W_t\}\),  with  filtration  \(\F_t = \sigma(B_s,W_s,\forall s\in[0,t])\).
The blue team's state processes \(\{V_t, Y_t\}\)  evolve under controls \(\{\alpha_t, \beta_t\}\) and are subject to Brownian noises \(\{B_t, W_t\}\):
\begin{align}
    &\ud V_t = \alpha_t\ud t + \sigma_B\ud B_t, \label{eqn:velocity}\\
    &\ud Y_t = (V_t + \beta_t)\ud t + \sigma_W\ud W_t.\label{eqn:position}
\end{align}
Here, \(\sigma_B, \sigma_W>0\) represent the volatility parameters governing the system dynamics. The initial conditions $V_0, Y_0 \in L^2(\Omega)$ are square-integrable random variables.
Without loss of generality, both the state and control processes are assumed to take values in $\R$.

The primary task over horizon $[0,T]$ is
\begin{equation}
    J^{\text{primary}}(\alpha,\beta) := \E\Big[\int_0^T r(t,V_t,Y_t,\alpha_t,\beta_t)\ud t + g(V_T,Y_T)\Big],
    \label{eqn:primary_cost}
\end{equation}
where states \(v, y\in\R\) and controls \(\alpha, \beta\in\R\) define the running and terminal costs:
\begin{align}
    \label{eqn:running_cost}
    r(t,v,y,\alpha,\beta) &= \frac{r_\alpha}{2}\alpha^2 + \frac{r_\beta}{2}\beta^2 + \frac{r_v}{2}(v-\overline{v}(t))^2,\\
    \label{eqn:terminal_cost}
    g(v,y) &= \frac{t_v}{2}(v-\overline{v}_T)^2.
\end{align}
The parameters \(r_\alpha\), \(r_\beta\), \(r_v\), \(t_v\) are strictly positive, and \(\BAR{v}_T \in \R\). The function \(\BAR{v}:[0,T]\to \R\) is given and continuous.


Regarding the information set on which the blue team bases its decisions, we consider Markovian control \cite{pham2009continuous}, meaning that the control processes follow the feedback form: 
\begin{equation}\label{eq:feedback}
    \alpha_t = \phi^\alpha(t,V_t,Y_t), \quad \beta_t = \phi^\beta(t,V_t,Y_t),
\end{equation}
where \(\phi^\alpha,\phi^\beta:[0,T]\times \R\times\R\to \R\) are Borel-measurable feedback functions. 

\begin{remark}[Model interpretation]\label{rem:model_interp}
    For the blue team's primary task \eqref{eqn:velocity}--\eqref{eqn:terminal_cost}, \eqref{eqn:velocity} models velocity $V_t$, controlled by acceleration \(\alpha_t\). Equation~\eqref{eqn:position} governs position  $Y_t$, with \(\beta\) providing (in addition to $V_t$) instantaneous velocity adjustment. The control \(\beta\) plays a crucial role in the strategic misdirection discussed later. The cost functionals~\eqref{eqn:running_cost}--\eqref{eqn:terminal_cost} reflect the blue team's primary objectives of keeping $V_t$ near the reference path \(\BAR{v}(t)\), reaching target velocity \(\BAR{v}_T\) at time $T$, and minimizing the control effort \(\frac{r_\alpha}{2}\alpha^2 + \frac{r_\beta}{2}\beta^2\).

    As shown later in Corollary~\ref{cor:baseline}, the optimal control for $J^{\text{primary}}$ is \(\hat{\beta} \equiv 0\): the blue team has no incentive to change its velocity instantaneously unless it expects the red team to engage in SHT. This insight is central to the formulation \eqref{eqn:hypotheses} of $H_0$ in Sec.~\ref{sec:IIB}.
    
    
\end{remark}

\subsection{Intention Inference with SHT in a Partially Observable Environment}\label{sec:IIB}

With the red team, a potential adversary, attempting to infer its intentions, the blue team adopts deceptive actions, modeled by $\beta_t$ in \eqref{eqn:position}. We first describe the partially observable setting and the red team's information set, then formulate how this leads to a secondary task that determines the optimal $\hat{\beta}_t$. Importantly, \textit{the red team considered here is the version perceived by the blue team, i.e., its belief about the red team’s knowledge and behavior.} The actual red team's information set will be discussed in Sec.~\ref{sec:game}.

From the blue team’s perspective, the red team knows the state dynamics~\eqref{eqn:velocity}--\eqref{eqn:position}, but not the primary task~\eqref{eqn:running_cost}--\eqref{eqn:terminal_cost}. The red team can observe sample paths of \(\{Y_t\}\), but cannot identify the underlying sample point \(\omega\in\Omega\) that corresponds to the currently observed trajectory of \(\{Y_t\}\). 
Hence, the sample paths of \(\{B_t,W_t,V_t\}\) remain unobservable to the red team, leading to a partially observable setting. 

Believing that the red team will attempt to infer intentions from limited observations, the blue team forms its own belief about the red team’s inference process, modeled by SHT. To anticipate the red team’s potential conclusions, the blue team envisions itself in the red team’s position, introduces a secondary task and selects $\beta_t$ strategically to reduce the likelihood of a conclusive or unfavorable inference by the red team.
For tractability, we assume linear feedback controls, motivated by the LQ structure of \eqref{eqn:velocity}--\eqref{eqn:terminal_cost}.

\begin{assu}[Control]
    \label{assu:linear_ctrl}
    Assume the feedback functions in~\eqref{eq:feedback} are linear in the state variables:
    \begin{align}
        \phi^\alpha(t, v, y) &= b_\alpha(t)v + c_\alpha(t)y + d_\alpha(t),\\
        \phi^\beta(t, v, y) &= b_\beta(t)v + c_\beta(t) y + d_\beta(t),
    \end{align}
    where the coefficients belong to \(C_T:= C([0,T];\R)\), the space of continuous real-valued functions on $[0,T]$.
\end{assu}

From the blue team's perspective, SHT is conducted with the following null and alternative hypotheses:
\begin{equation}
    \label{eqn:hypotheses}
    \begin{cases}
    H_0: b_\beta \equiv 0, c_\beta \equiv 0, d_\beta \equiv 0 \\
    H_1: b_\beta \equiv 0, c_\beta = f_c, d_\beta = f_d 
    \end{cases},
\end{equation}
where \(f_c, f_d  \in C_T\). By choosing \(f_c, f_d\), the blue team can shape patterns that mislead the red team. Rejecting \(H_0\) means a significant likelihood to engage in strategic misdirection.

\begin{remark}
\label{rem:hyp_interp}

 The red team assumes \(b_\beta \equiv 0\) due to its inability to observe the sample paths of \(\{V_t\}\). The null hypothesis \(H_0\) corresponds to the baseline case where $\hat\beta \equiv 0$, as discussed later in Corollary~\ref{cor:baseline}. Under either hypothesis $H_0$ or $H_1$, \(\E V_t^2, \E Y_t^2<\infty\) for \(t\in[0,T]\), with \(\{V_t, Y_t\}\) having almost surely (a.s.) continuous paths.  Extensions such as nonlinear dependence or inference on $\phi^\alpha$ are natural but beyond this work, thus are left for future research.


\end{remark}

Statistically, test~\eqref{eqn:hypotheses} is a simple {\it vs.} simple test, with parameter space \((C_T)^{3}\), where \(H_0\) and \(H_1\) correspond to \(\Theta_0:= \{(0,0,0)\}\) and \(\Theta_1:= \{(0,f_c,f_d)\}\).
Given this structure, the sequential probability ratio test (SPRT) (\cite{tartakovsky2014sequential,goodman2007adaptive,schonbrodt2017sequential}) is the natural choice, as it is the most powerful sequential test, as indicated by Neyman–Pearson-type results \cite{wald1948optimum}. 
In this context, SHT refers specifically to SPRT.

To construct the test statistic, let \((C_T,\mathscr{B}_T)\) be a measurable space with $\sigma$-field \(\mathscr{B}_T\).
For stochastic processes \(\{\eta_t,\xi_t\}\) with a.s. continuous paths, denote their laws by $\mu_\eta,\mu_\xi$, given by \(\mu_\eta(B) := \Prob(\eta\in B)\), \(\forall B\in\mathscr{B}_T\). Whenever \(\mu_\eta\) is absolutely continuous with respect to \(\mu_\xi\), the Radon-Nikodym derivative \(\frac{\ud \mu_\eta}{\ud \mu_\xi}:C_T\to\R_+\) is well-defined, and serves as the (likelihood ratio) statistic of SHT.

\begin{prop}
    \label{prop:LR_calc}
    Denote by \(\mu^{H_0}_{(V,Y)}\) and \(\mu^{H_1}_{(V,Y)}\) the law of \(\{V_t,Y_t\}\) under \(H_0\) and \(H_1\) respectively.
    Under Assumption~\ref{assu:linear_ctrl}, the SHT statistic is given by:
    \begin{align}
        \label{eqn:L_t}
        L_T(V,Y) :&= \frac{\ud \mu^{H_1}_{(V,Y)}}{\ud \mu^{H_0}_{(V,Y)}}(V,Y)\notag\\
        &= \text{exp}\Bigg\{\frac{1}{\sigma_W^2}\Bigg[
        \int_0^T \left(f_c(t)Y_t + f_d(t)\right)\ud Y_t\notag\\
        &- \int_0^T V_t\left(f_c(t)Y_t + f_d(t)\right)\ud t \notag\\
        &- \frac{1}{2}\int_0^T \left(f_c(t) Y_t + f_d(t)\right)^2 \ud t    \Bigg]\Bigg\},
    \end{align}
    where \((V,Y)\) are random processes that follow the dynamics~\eqref{eqn:velocity}--\eqref{eqn:position} under \(H_0\). 
\end{prop}

\begin{proof}
Using dynamics~\eqref{eqn:velocity}--\eqref{eqn:position} and interpreting \(\{V_t,Y_t\}\) as SDEs with different drifts under $H_0$ and $H_1$, the SHT statistic $L_T$ follows directly from \cite[Sec.~7.6]{liptser2013statistics}. 
For a more detailed proof, refer to \cite[Lemma~1 \& Proposition~1]{zhou2025adversarial}.
\end{proof}

The SHT statistic \(L_T(V,Y)\) is measurable with respect to the $\sigma$-field \(\sigma(V_t,Y_t,\forall t\in[0,T])\). Though $V$ is unobservable to the red team, it is available to the blue team, who uses it to plan misdirection. 
A large $L_T$ favors rejecting $H_0$. In practice, given sample paths $(\tilde{v},\tilde{y})$, the test statistic shall be evaluated as  $L_T(\tilde{v},\tilde{y})$.

\subsection{Linear-Quadratic Model for Strategic Misdirection}\label{sec:IIC}

With the SHT (likelihood ratio) statistic $L_T$ established, we now introduce the blue team’s bi-objective optimization problem, integrating the primary task from Sec.~\ref{sec:IIA}. We derive a semi-explicit solution for the optimal strategy by reducing the problem to a system of ODEs and proving global existence and uniqueness of the solution.

In the blue team’s perspective, the envisioned red team conducts SHT to detect any perturbations in the sample paths of \(\{Y_t\}\), aiming to infer its true intentions. Anticipating this, the blue team deliberately perturbs $Y_t$, even at the expense of partially sacrificing its primary task. Based on the interpretation of $L_T$ and Proposition~\ref{prop:LR_calc}, the blue team maximizes \(\log L_T\),  leading it to add the term \(-\E \log L_T\) to its primary cost~\eqref{eqn:primary_cost}, as calculated in Proposition~\ref{prop:E_log_L}.

\begin{prop}
    \label{prop:E_log_L}
    Under Assumption~\ref{assu:linear_ctrl}, the expected log-likelihood ratio statistic, evaluated at the empirically observed trajectories \((V,Y)\), is given by:
    \begin{multline}
        \label{eqn:E_log_Lt}
        \E \log L_T = \frac{1}{\sigma_W^2}\E\int_0^T \Big[(f_c(t)Y_t  + f_d(t))\beta_t \\- \frac{1}{2}(f_c(t)Y_t +f_d(t))^2 \Big]\ud t.
    \end{multline} 
\end{prop}

\begin{proof}
    Define \(Z_t := \int_0^t (f_c(s)Y_s + f_d(s))\ud W_s\). Since \(\E \QV{Z}{Z}_T<\infty\), it follows that \(\{Z_t\}\) is a martingale with zero mean.
    Taking logarithm and expectation on both sides of \eqref{eqn:L_t} and substituting the dynamics~\eqref{eqn:position} yield
    \begin{multline}
        \E \log L_T = \frac{1}{\sigma_W^2} \E\Big[ \int_0^T\left(f_c(t)Y_t + f_d(t)\right)\beta_t\ud t\\
        - \frac{1}{2}\int_0^T \left(f_c(t) Y_t + f_d(t)\right)^2 \ud t + \sigma_W Z_T\Big],
    \end{multline}
    which concludes the proof.
\end{proof}

Considering both the primary task~\eqref{eqn:primary_cost} and strategic misdirection~\eqref{eqn:E_log_Lt}, the blue team now aims to minimize:
\begin{equation}\label{def:J}
    J_{\text{blue}}(\alpha,\beta) := J^{\text{primary}}(\alpha, \beta) - \lambda\E\log L_T,
\end{equation}
where \(0\leq \lambda\leq r_\beta\sigma_W^2\) measures the intensity of misdirection. The upper bound of \(\lambda\) ensures well-posedness (Theorem~\ref{thm:global_well}).

This new stochastic control problem retains the state dynamics~\eqref{eqn:velocity}--\eqref{eqn:position} and terminal cost~\eqref{eqn:terminal_cost}, but modifies the running cost (c.f. \eqref{eqn:running_cost}) into: 
\begin{multline}
    \label{eqn:mod_running_cost}
    h(t,v,y,\alpha,\beta) := r(t,v,y,\alpha,\beta)\\
    - \frac{\lambda}{\sigma^2_W}(f_c(t)y + f_d(t)) \beta + \frac{\lambda}{2\sigma_W^2}(f_c(t)y + f_d(t))^2,
\end{multline}
which incorporates the integrands from \eqref{eqn:E_log_Lt}.

The remainder of Sec.~\ref{sec:IIC} solves this LQ control problem~\eqref{def:J}. Let \(\mathcal{V}(t, v, y):[0,T]\times \R\times \R\to\R\) be the value function, representing the minimized cost when the system starts at $(V_t, Y_t) = (v, y)$. 
By dynamic programming, $\mathcal{V}$ satisfies the Hamilton-Jacobi-Bellman (HJB) equation:
\begin{multline}
    \partial_t \mathcal{V} + \inf_{\alpha,\beta}\{\alpha\partial_v \mathcal{V} + (v+\beta)\partial_y \mathcal{V} + \frac{1}{2}\sigma_B^2\partial_{vv}\mathcal{V} \\+ \frac{1}{2}\sigma_W^2\partial_{yy}\mathcal{V}  
    + h(t,v,y,\alpha,\beta)\} = 0,
    \label{eqn:HJB}
\end{multline}
with a terminal condition $
    \mathcal{V}(T,v,y) = \frac{t_v}{2}\big(v - \overline{v}_T\big)^2$.
Minimizing over $\alpha,\beta$ yields the optimal controls: 
\begin{align}
    \label{eqn:alpha_hat}
    \hat{\alpha} &=  -\frac{1}{r_\alpha}\partial_v \mathcal{V},\\
    \label{eqn:beta_hat}
    \hat{\beta} &= \frac{1}{r_\beta}\Big[ \frac{\lambda}{\sigma_W^2}\Big(f_c(t)y + f_d(t)\Big) - \partial_y \mathcal{V}\Big].
\end{align}
Adopting a quadratic ansatz, we assume
\begin{equation}
    \mathcal{V}(t,v,y) = \frac{\mu_t}{2}v^2 + \eta_tvy + \frac{\rho_t}{2}y^2 + \gamma_t v + \theta_t y +  \xi_t,
\end{equation}
where \(\mu\), \(\eta\), \(\rho\), \(\gamma\), \(\theta\), \(\xi \in C_T\) .
Plugging \eqref{eqn:alpha_hat} and~\eqref{eqn:beta_hat} back into \eqref{eqn:HJB} and collecting corresponding coefficients yields a system of ODEs:
\begin{equation}
    \label{eqn:ODE_system}
    \begin{cases}
        \dot{\mu}_t = \frac{1}{r_\alpha}\mu_t^2 + \frac{1}{r_\beta}\eta_t^2 - 2\eta_t - r_v,\\
        \dot{\eta}_t = \frac{1}{r_\alpha}\mu_t \eta_t + \frac{1}{r_\beta}\rho_t \eta_t - \rho_t - \frac{\lambda}{r_\beta \sigma_W^2}\eta_t f_c(t),\\
        \dot{\rho}_t = \frac{1}{r_\alpha}\eta_t^2 + \frac{1}{r_\beta} \rho_t^2 - \frac{2\lambda}{r_\beta \sigma_W^2}\rho_t f_c(t) + (\frac{\lambda^2}{r_\beta \sigma_W^4} - \frac{\lambda}{\sigma_W^2})f_c^2(t),\\
        \dot{\gamma}_t = \frac{1}{r_\alpha} \mu_t \gamma_t + \frac{1}{r_\beta}\eta_t \theta_t - \theta_t + r_v \overline{v}(t) - \frac{\lambda}{r_\beta \sigma_W^2}\eta_t f_d(t),\\
        \dot{\theta}_t = \frac{1}{r_\alpha}\eta_t \gamma_t + \frac{1}{r_\beta}\rho_t \theta_t - \frac{\lambda}{r_\beta \sigma_W^2}\theta_t f_c(t) - \frac{\lambda}{r_\beta \sigma_W^2}f_d(t) \rho_t \\\quad\quad + (\frac{\lambda^2}{r_\beta \sigma_W^4} - \frac{\lambda}{\sigma_W^2})f_c(t) f_d(t),\\
        \dot{\xi}_t = \frac{1}{2r_\alpha}\gamma_t^2 + \frac{1}{2r_\beta}\theta_t^2 - \frac{1}{2}\sigma_B^2 \mu_t - \frac{1}{2}\sigma_W^2 \rho_t - \frac{\lambda}{r_\beta \sigma_W^2}f_d(t) \theta_t \\
        \quad \quad- \frac{r_v[\overline{v}(t)]^2}{2} + (\frac{\lambda^2}{2 r_\beta \sigma_W^4} - \frac{\lambda}{2 \sigma_W^2})f_d^2(t),
    \end{cases}
\end{equation}
with terminal conditions $\mu_T = t_v$, $\eta_T = 0$, $\rho_T = 0$, $\gamma_T = -t_v\overline{v}_T$, $\theta_T = 0$, $\xi_T = \frac{t_v}{2}\big(\overline{v}_T\big)^2$.
The semi-explicit solution to this linear-quadratic control problem is given by
\begin{equation}
    \label{eqn:solution}
    \begin{aligned}
        \hat{\alpha}(t,v,y) = &-\frac{\mu_t}{r_\alpha} v - \frac{\eta_t}{r_\alpha} y - \frac{\gamma_t}{r_\alpha},\\
        \hat{\beta}(t,v,y) = &-\frac{\eta_t}{r_\beta} v + \Big(\frac{\lambda}{r_\beta\sigma_W^2}f_c(t) - \frac{\rho_t}{r_\beta}\Big) y \\ &+ \Big(\frac{\lambda}{r_\beta\sigma_W^2}f_d(t) - \frac{\theta_t}{r_\beta}\Big).
    \end{aligned}
\end{equation}
This solution confirms Assumption~\ref{assu:linear_ctrl} and illustrates optimal deception strategies. While $\alpha$ and $\beta$ are intended to handle primary and secondary objectives separately, the optimal solution couples them, so that $\hat\alpha$ and $\hat\beta$ differ from controls obtained if tasks were solved independently. Comparison of $\hat\beta$ with $H_1$ in \eqref{eqn:hypotheses} shows that the blue team misdirects toward patterns induced by $f_c$ and $f_d$. Philosophically, introducing perturbations is meaningful only when constrained to specific patterns -- justifying the use of a simple-vs-simple SHT rather than general information divergences. Notably, $\hat\alpha$ and $\hat\beta$ depend nonlinearly on $f_c,f_d$, since changes propagate across the coupled ODE system  \((\mu_t, \eta_t, \rho_t, \gamma_t, \theta_t)\).

We now show the well-posedness of the ODE system~\eqref{eqn:ODE_system}. Denoting \(x(t) := (\mu_t, \eta_t, \rho_t, \gamma_t, \theta_t, \xi_t)\transpose\), the system~\eqref{eqn:ODE_system} is \(\dot{x}(t) = F(t,x(t))\), where \(F:[0,T]\times \R^6\to\R^6\), and  \(x_i(t)\) is the \(i\)-th component of \(x(t)\). 



\begin{theorem}[Local Existence and Uniqueness]
    \label{thm:local_well}
    There exists \(t_0\in\R_+\) and \(m>0\) such that the solution \(x(t)\) to the ODE system~\eqref{eqn:ODE_system} exists and is unique, when \(t\in[T-t_0,T]\) and \(\|x(t) - x(T)\|_\infty\leq m\).
\end{theorem}

\begin{proof}
    Since continuous functions are bounded on compact sets, $\exists M_f>0$, such that \(\ \max\{|f_c(t)|, |f_d(t)|,|\BAR{v}(t)|\}\leq M_f\), \(\forall t\in[0,T]\).
    Setting \(M:= m + \|x(T)\|_\infty\), we have \(\|x(t)\|_\infty\leq M\), \(\forall t\in[T-t_0,T]\).

    Applying the triangle inequality to \(\|F(t,x) - F(t,y)\|_1\) yields three types of terms:
    quadratic terms \(|x_i^2 - y_i^2|\leq 2M|x_i-y_i|\);
    crossing terms \(|x_ix_j - y_iy_j|\leq M|x_i-y_i| + M|x_j-y_j|\), for \(i\neq j\); and  linear terms \(|x_i-y_i|\).
    Collecting the coefficients gives 
    $\|F(t,x) - F(t,y)\|_1\leq \sum_{i=1}^6 C_i|x_i-y_i|\leq \max_i |C_i|\cdot \|x-y\|_1,$
    showing that \(F\) is Lipschitz continuous in \(x\) under the \(\ell_1\) norm \(\|\cdot\|_1\).

    Since \(F\) is continuous in \(t\) and Lipschitz in \(x\), the Picard-Lindelöf theorem \cite{teschl2012ordinary} guarantees local existence and uniqueness, concluding the proof.
\end{proof}



We recall the following lemma on the {\it a priori} bound of the solution to a matrix Riccati equation.

\begin{lemma}[{\cite[Sec.~2]{jacobson1970new}, \cite[Sec.~3]{mrtensson1971matrix}}]
\label{lem:Riccati}
Denote by \(\mathbb{S}^{n\times n}\) the collection of \(n\times n\) real symmetric matrices.
Consider the matrix Riccati equation:
\begin{equation}
    \dot{S}(t) + Q(t) + A(t)\transpose S(t) + S(t)A(t)  - S\transpose (t)R^{-1}(t)S(t)= 0,
\end{equation}
for \(S:[0,T]\to \mathbb{S}^{n\times n}\), with a given terminal condition \(S(T)\), where \(Q,R^{-1}:[0,T]\to \mathbb{S}^{n\times n}\) and \(A:[0,T]\to \mathbb{R}^{n\times n}\) are continuous matrix-valued functions. 
If the following conditions hold:
\begin{equation}
    \label{eqn:Riccati_condition}
    \begin{cases}
        Q(t) \geq 0,\ R^{-1}(t)>0,\ \forall t\in[0,T]\\
        S(T)\geq 0
    \end{cases},
\end{equation}
then the solution \(S(t)\) has an {\it a priori} bound on \([0,T]\) that only depends on the coefficients and terminal conditions of the matrix Riccati equation.
\end{lemma}

\begin{theorem}[Global Existence and Uniqueness]
    \label{thm:global_well}
    Under the condition \(0\leq\lambda \leq r_\beta \sigma_W^2\), the ODE system~\eqref{eqn:ODE_system} has a unique global solution on \([0,T]\), for any \(T>0\).
\end{theorem}

\begin{proof}
    From Theorem~\ref{thm:local_well}, global existence and uniqueness hold if the solution \(x(t)\) is {\it a priori} bounded on \([0,T]\).

We first prove the {\it a priori} boundedness of \(\mu_t\), \(\eta_t\), \(\rho_t\). Rewriting \eqref{eqn:ODE_system} in terms of
    \begin{align}
    S(t) &= \begin{bmatrix}
            \mu_t & \eta_t\\
            \eta_t & \rho_t
        \end{bmatrix},\quad
        S(T) = \begin{bmatrix}
           t_v & 0\\
           0 & 0
        \end{bmatrix},\\
        Q(t) &= \begin{bmatrix}
            r_v & 0\\
            0 & -\Big(\frac{\lambda^2}{r_\beta \sigma_W^4} - \frac{\lambda}{\sigma_W^2}\Big)f_c^2(t)
        \end{bmatrix},\\
        R^{-1}(t) &= 
        \begin{bmatrix}
            \frac{1}{r_\alpha} & 0\\
            0 & \frac{1}{r_\beta}
        \end{bmatrix},
        \quad 
        A(t) = \begin{bmatrix}
            0 & 0\\
            1 & \frac{\lambda}{r_\beta \sigma_W^2}f_c(t)
        \end{bmatrix}.
    \end{align}
    The coefficients are continuous in \(t\) and satisfy the conditions~\eqref{eqn:Riccati_condition}.
    By Lemma~\ref{lem:Riccati}, \(\mu_t\), \(\eta_t\), \(\rho_t\) are {\it a priori} bounded on \([0,T]\). Then, since \(\gamma_t\), \(\theta_t\) satisfy linear ODEs with bounded time-variant coefficients, their solutions remain {\it a priori} bounded. 
    Finally, the {\it a priori} boundedness of \(\xi_t\) follows from the integral representation in terms of {\it a priori} bounded functions \(\mu_t\), \(\eta_t\), \(\rho_t\), \(\gamma_t\), and \(\theta_t\).
    This concludes the proof.
\end{proof}



With global well-posedness established, we now examine cases when strategic misdirection is absent or SHT is trivial. As shown in Corollary~\ref{cor:baseline}, the optimal control $\hat\beta$ remains trivial when either misdirection is excluded $(\lambda =0)$ or the hypotheses in~\eqref{eqn:hypotheses} coincide $(H_0 = H_1)$. In these cases, the problem reduces to the blue team’s primary task, and the optimal strategy introduces no perturbations.

\begin{cor}
    \label{cor:baseline}
    Under the condition \(0\leq\lambda \leq r_\beta \sigma_W^2\), if \(f_c \equiv f_d \equiv 0\) or \(\lambda=0\), then the unique solution to the ODE system~\eqref{eqn:ODE_system} satisfies \(\eta\equiv\rho\equiv\theta\equiv 0\), implying that \(\hat{\beta} \equiv 0\).
\end{cor}

\begin{proof}
    When \(f_c\equiv f_d\equiv 0\) or \(\lambda = 0\), the ODEs governing \(\rho_t\) and \(\theta_t\) become homogeneous, which implies that \(\eta\equiv \rho\equiv \theta\equiv 0\) is a solution to system~\eqref{eqn:ODE_system}.
    By Theorem~\ref{thm:global_well}, this solution is unique.
    Substituting their values into \eqref{eqn:solution}, we conclude that \(\hat{\beta}\equiv 0\).
\end{proof}

\section{The Stackelberg Game}\label{sec:game}

This section examines how the actual red team (as opposed to the envisioned one) strategically counters the blue team's misdirection. For technical simplicity, we present a simplified LQ model from Sec.~\ref{sec:IIC}, fixing \(f_d\equiv \BAR{v}\equiv 0\) and \(\BAR{v}_T = 0\), which implies $\gamma \equiv \theta \equiv 0$. This reduced model preserves the essence of strategic misdirection but represents it entirely through $f_c$.

The actual red team is assumed to know the governing dynamics~\eqref{eqn:velocity}--\eqref{eqn:position} but can only observe the sample paths of \(\{Y_t\}\). 
However, it has agents capable of:
\begin{itemize}
    \item Leaking the blue team's current misdirection choice \(f_c\) and the functional forms of its strategies \(\hat{\alpha},\hat{\beta}\) in \eqref{eqn:solution} (but not the ODEs governing $(\mu_t,\eta_t,\rho_t,\gamma_t,\theta_t,\xi_t)$)\footnote{Without prior knowledge of the LQ structure, the red team cannot uniquely determine the blue team’s cost functionals. This is known as the inverse reinforcement learning (IRL) problem, where a given control strategy can be optimal under multiple cost functionals \cite{ng2000algorithms}.}.
    \item Imposing a new pattern \(f_c\), which the blue team accepts.
\end{itemize}
Recognizing that the observed sample paths of \(\{Y_t\}\) contain blue team-induced perturbations, the red team aims to lure the blue team into voluntarily reducing its misdirection efforts while minimizing suspicion about its agents. 

\smallskip
\noindent\textbf{Stackelberg Game Formulation.} The red team leads the game by optimizing $f_c$, after which the blue team solves its LQ problem with SHT, under the instilled $\hat f_c$, as described in Sec.~\ref{sec:model}. This interaction forms a Stackelberg game, where the red team’s optimization problem consists of two components:
\begin{itemize}
    \item \textbf{Misdirection Control}: minimize $\E\log L_T$, using~\eqref{eqn:solution}, to push the blue team’s behavior toward $H_0$.  Confirming $H_0$ in~\eqref{eqn:hypotheses} means the blue team introduces no perturbations in $\{Y_t\}$.

    \item \textbf{Regularization}: add a penalty $\mathcal{P}(f_c)$ to model the blue team’s skepticism toward the manipulated $f_c$. Without this, minimization of $\E\log L_T$ tends to produce the trivial solution aligning exactly with $H_0$ \cite[Theorem~3]{zhou2025adversarial}, which the blue team would reject to use as it provides no misdirection.
    
    
\end{itemize}
Therefore, the red team's objective is formulated as:
    \begin{equation}\label{def:Obj-red}
        J_{\text{red}}(f_c) := \E[\log \hat{L}_T] + \frac{\lambda_{\text{reg}}}{\sigma_W^2} \mathcal{P}(f_c),
    \end{equation}
where $\hat{L}_T$ is defined in \eqref{eqn:L_t}, evaluated at $(\hat V, \hat Y)$, which follow \eqref{eqn:velocity}--\eqref{eqn:position}  under the optimal controls $(\hat\alpha, \hat \beta)$ provided in \eqref{eqn:solution}. 
Here $\lambda_{\text{reg}}>0$ characterizes the intensity of regularization.

Using \eqref{eqn:velocity}--\eqref{eqn:position}, \eqref{eqn:solution}, Itô's formula, and Fubini's theorem,
\begin{multline}
\label{eqn:red_min_obj_unreg}
   \E[\log \hat L_T] = \frac{1}{\sigma_W^2}\int_0^T \bigg[\left(-\frac{\eta_t}{r_\beta}h_{11}(t) - \frac{\rho_t}{r_\beta}h_{02}(t)\right)f_c(t) \\
    + \left(\frac{\lambda}{r_\beta \sigma_W^2}-\frac{1}{2}\right)h_{02}(t)f_c^2(t) \bigg]\ud t,
\end{multline}
where $h_{ij}(t) :=  \E [\hat{V}_t^i\hat{Y}_t^j]$ can be computed via the coupled ODE system:
\begin{equation}\label{eqn:ODE_moments}
    \begin{cases}
        \dot{h}_{20}(t) = -2\frac{\mu_t}{r_\alpha}h_{20}(t) - 2\frac{\eta_t}{r_\alpha}h_{11}(t) + \sigma_B^2, \\
        \dot{h}_{11}(t) = \left(\frac{\lambda}{r_\beta\sigma_W^2}f_c(t) - \frac{\rho_t}{r_\beta} - \frac{\mu_t}{r_\alpha}\right)h_{11}(t) \\
        \quad \quad\quad\quad + \left(1-\frac{\eta_t}{r_\beta}\right)h_{20}(t) - \frac{\eta_t}{r_\alpha}h_{02}(t), \\
        \dot{h}_{02}(t) = 2\left(1 - \frac{\eta_t}{r_\beta}\right)h_{11}(t) \\
        \quad \quad\quad\quad + 2\left(\frac{\lambda}{r_\beta\sigma_W^2}f_c(t) - \frac{\rho_t}{r_\beta}\right)h_{02}(t) + \sigma_W^2,
    \end{cases}
\end{equation}
with initial conditions
$  h_{20}(0) = \E V_0^2$, $h_{11}(0) = \E V_0Y_0$, $h_{02}(0) = \E Y_0^2$.

For the penalty term $\mathcal{P}(f_c)$, we impose a trust region constraint around the initial pattern $f_c^{\text{initial}}$. After the red team's manipulation, the blue team adopts the new $f_c$ but grows more skeptical in proportion to its deviation from $f_c^{\text{initial}}$. Such modeling ensures controlled, credible adjustments. We define
\begin{multline}\label{def:reg}
    \mathcal{P}[f_c] := \int_0^T [f_c(t)-f_c^{\text{initial}}(t)]^2 \ud t \\
    \text{ or }\int_0^T f_c^{\text{initial}}(t)\log \frac{f_c^{\text{initial}}(t)}{f_c(t)} \ud t.
\end{multline}
The first is quadratic, and the second, inspired by KL-divergence, treats \(f_c^{\text{initial}}\) and \(f_c\) as unnormalized densities on $[0,T]$. 

Since $ J_{\text{red}}(f_c)$ does not admit a closed-form minimizer, Sec.~\ref{sec:IV-B} explores three numerical approaches, including a neural network-based method, to compute the optimal $f_c$ and present the numerical results.




\section{Numerical Results}\label{sec:results}

In this section, we present numerical algorithms and experiments on deceptive and counter-deceptive strategies, focusing on the blue team's optimal misdirection  $\hat\beta_t$ (cf. Sec.~\ref{sec:model}) and the red team's optimal pattern $\hat f_c$ (cf. Sec.~\ref{sec:game}). {For algorithmic details and hyperparameters, see \cite{zhou2025adversarial}.}

\subsection{Blue Team's Optimal Control $\hat \beta_t$}\label{sec:IV-A}


We illustrate the blue team’s control problem from Sec.~\ref{sec:IIC} by plotting trajectories of velocity $\hat V_t$, observed position $\hat Y_t$, and optimal controls $\hat\alpha_t$, $\hat\beta_t$. All experiments use the parameter set
\begin{center}
    $T = 1, \sigma_B = \sigma_W = 0.25, V_0 = 2, Y_0 = 4, r_\alpha = 1,r_\beta = 10$\\
    $r_v = 1,t_v = 1,\BAR{v}_T = 1,\BAR{v}(t) = 2 - t,f_d\equiv 0.$
\end{center}
\vspace{-10pt}
\begin{figure}[!htbp]
    \centering
    \includegraphics[width=\linewidth, height = 0.22\textheight]{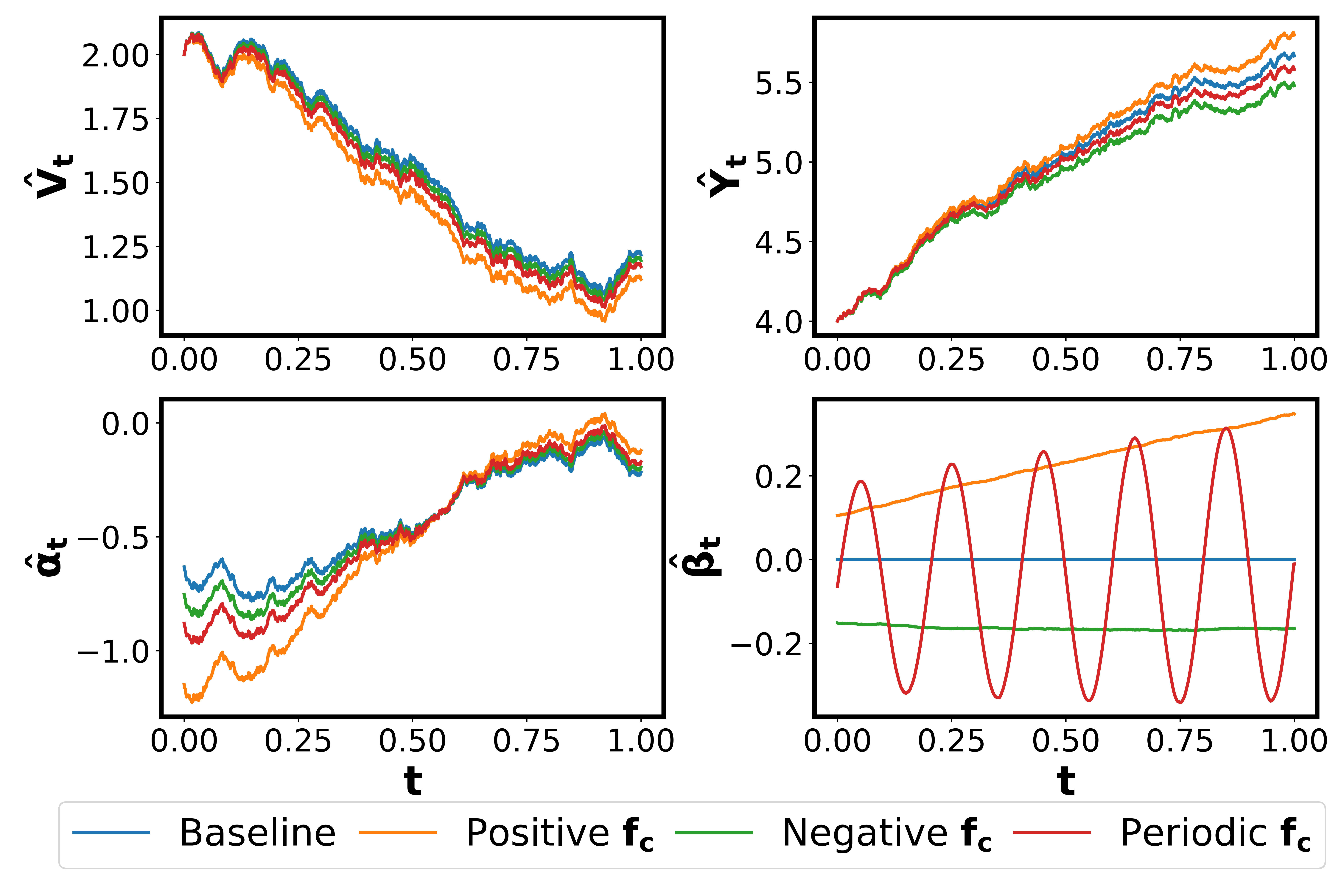}
    \caption{Comparisons of the optimal trajectories \eqref{eqn:velocity}--\eqref{eqn:position} and controls \eqref{eqn:solution} with \(\lambda = 0.075\) across different choices of \(f_c\): baseline \(f_c\equiv 0\), positive \(f_c\equiv 0.5\), negative \(f_c\equiv -0.25\), and periodic \(f_c(t) = 0.5\sin(10\pi t)\).
    }\label{fig:blue_test1}
\end{figure}
\vspace{-5pt}
Fig.\ref{fig:blue_test1} compares different $f_c$ choices. The baseline $f_c\equiv0$ corresponds to no perturbation (cf. Cor.\ref{cor:baseline}); a positive constant shifts the trajectory upward, a negative constant shifts it downward, and a periodic $f_c$ creates oscillatory deviations that can obscure intent in dynamic settings.

\begin{figure}[!htbp]
    \centering
    \includegraphics[width=\linewidth, height = 0.22\textheight]{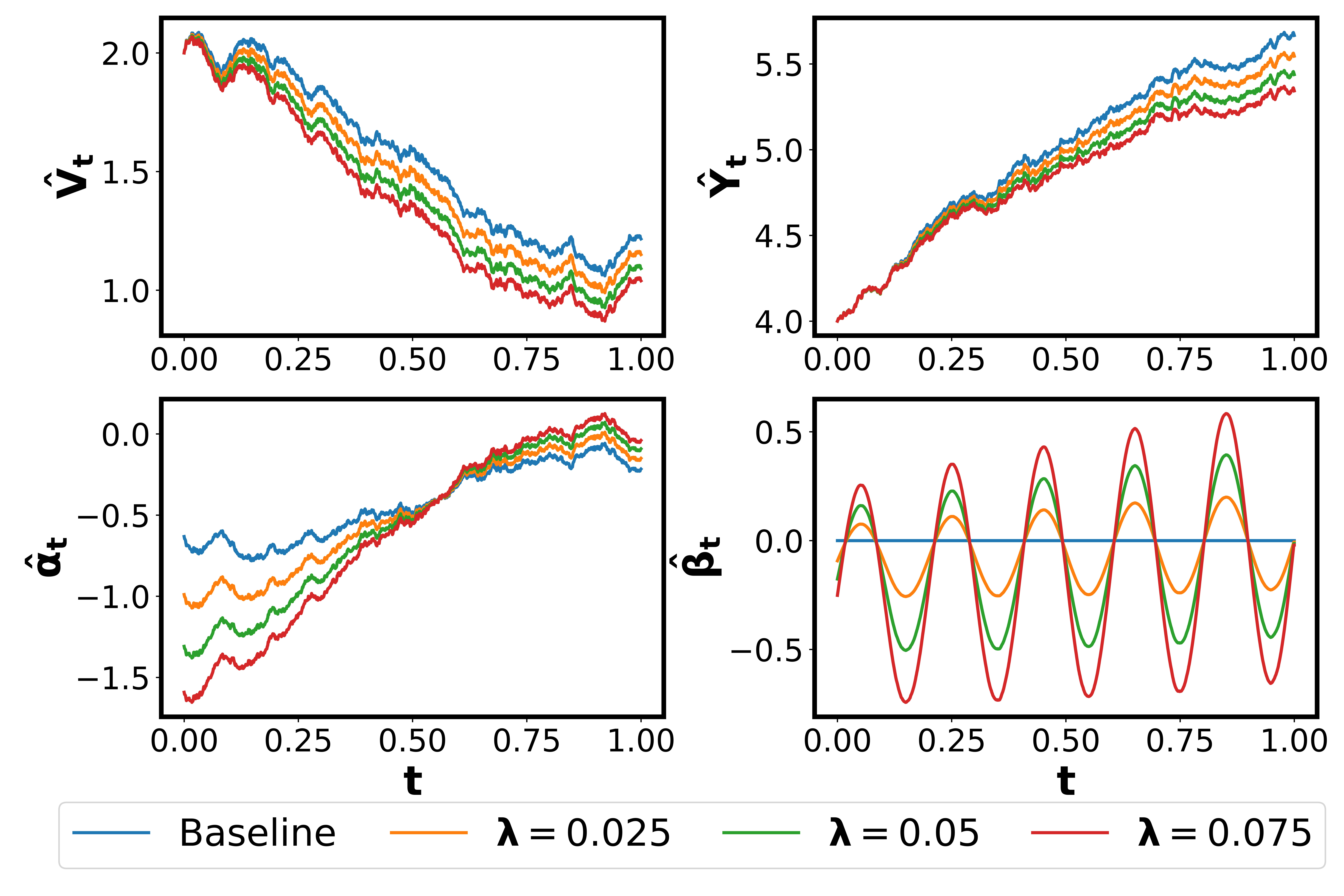}
    \caption{Comparisons of the optimal trajectories \eqref{eqn:velocity}--\eqref{eqn:position} and controls \eqref{eqn:solution} with \(f_c(t) = \sin(10\pi t)\) across different values of \(\lambda\).}
    \label{fig:blue_test2}
\end{figure}

Fig.~\ref{fig:blue_test2} varies the misdirection intensity $\lambda$ with periodic $f_c$. As $\lambda$ increases, the blue team emphasizes deception more strongly, producing larger deviations from the baseline. Monte Carlo simulations (10,000 paths) quantify the trade-off: $J^{\text{primary}}(\hat\alpha,\hat\beta)$ increases from $0.33$ at $\lambda=0$ to $0.44$ at $\lambda = 0.025$, $0.78$ at $\lambda = 0.05$ and $1.32$ at $\lambda=0.075$, while $\E[\log \hat L_T]$ improves from $-96.98$ to $-87.26, -78.14, -69.55$. Thus, a higher $\lambda$ enhances deception but reduces efficiency in accomplishing the primary task.

\subsection{Red Team's Optimal Control $\hat f_c$}\label{sec:IV-B}

To compute the red team's regularized optimal control $\hat f_c$ in \eqref{def:Obj-red}, we examine three algorithms: fixed point iteration (FPI), a neural network-based method  (NN), and the forward-backward sweep method (FBS). Experiments are run for both regularization choices in \eqref{def:reg}.

FPI starts with an initial guess $f_c^{(0)}$. At the $i^{th}$ iteration, the systems  \eqref{eqn:ODE_system} and \eqref{eqn:ODE_moments} are solved with $f_c$ replaced by $f_c^{(i)}$. The optimization in \eqref{def:Obj-red} is then performed using the obtained $(\mu_t, \eta_t, \rho_t)$ and $(h_{20}, h_{11}, h_{02})$, yielding $f_c^{(i+1)}$. The process iterates until convergence.

In the NN method \cite{han2016deep}, $f_c$ is parameterized by a feed-forward network with $t\in[0,T]$ as input. At each step, ODEs \eqref{eqn:ODE_system} and \eqref{eqn:ODE_moments} are solved by Euler schemes, and $J_{\text{red}}$ serves as the loss for training, which repeats until the network converges under a suitable norm.


The FBS method \cite[Ch.4]{lenhart2007optimal} leverages Pontryagin’s maximum principle by forming a Hamiltonian with costate (adjoint) equations. It alternates between solving the state and costate ODEs, and minimizing the Hamiltonian with respect to $f_c$. Convergence is guaranteed under the Lipschitz conditions of the ODEs \cite{mcAsey2012convergence}.


Fig.~\ref{fig:redtest9&11} presents the optimal $\hat f_c$ computed using different algorithms and penalty terms $\mathcal{P}(f_c)$. The following parameter values are used \footnote{Such choices align with the simplified LQ model, in which we have set \(f_d\equiv \BAR{v}\equiv 0\) and \(\BAR{v}_T = 0\).}:
\begin{center}
$T = 0.1, \sigma_B = \sigma_W = 0.1, V_0 = 1, Y_0 = 2, r_\alpha = 1,r_\beta = 10$\\
    $r_v = 1,t_v = 1,\BAR{v}_T = 0,\BAR{v}(t) \equiv 0,f_d\equiv 0,f_c^{\text{initial}} \equiv 1.$
\end{center}

\begin{figure}[!htbp]
    \centering
    \includegraphics[width=1.0\linewidth]{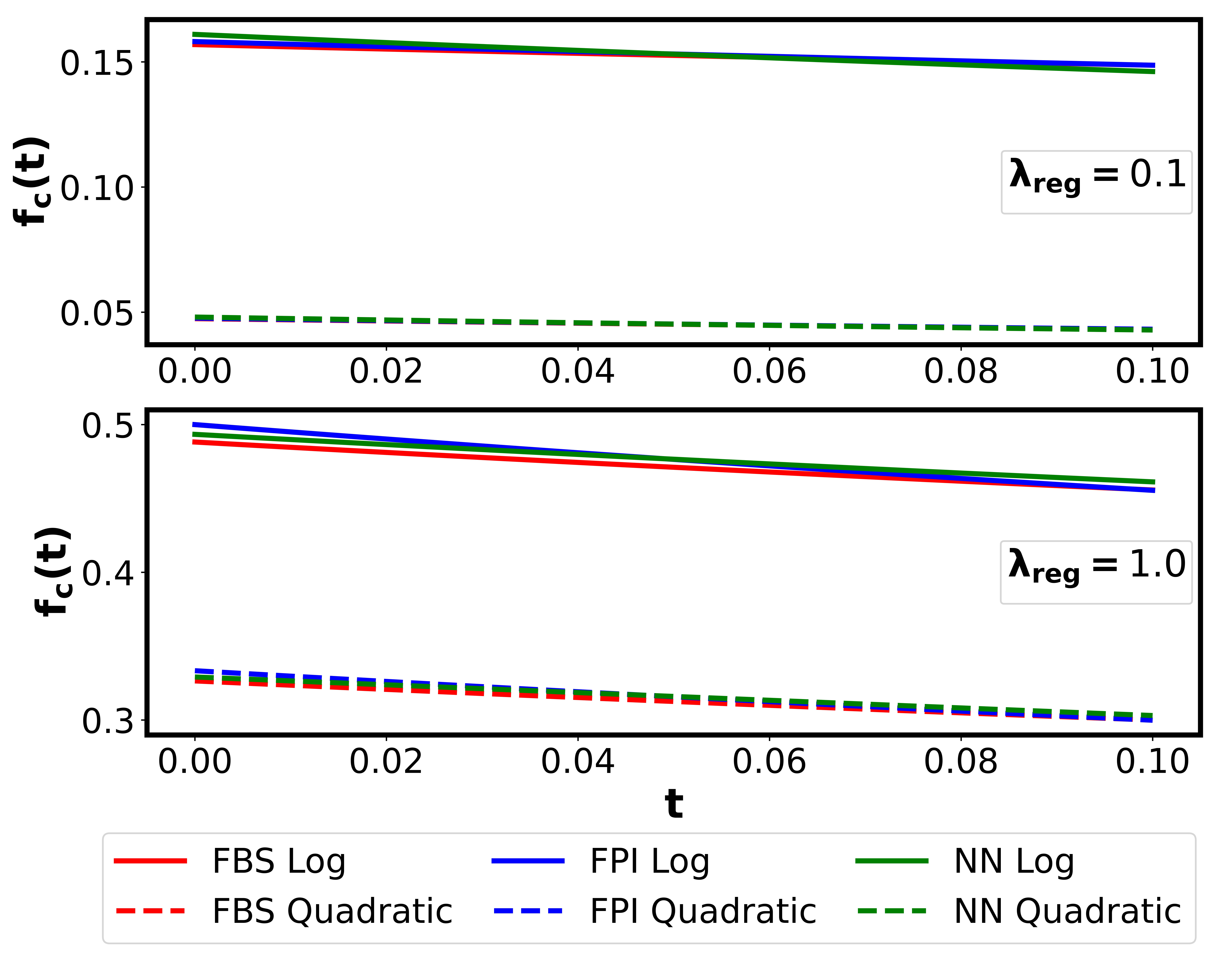}
    \caption{Plots of the optimal $\hat f_c$ across different algorithms, penalties, and values of \(\lambda_{\text{reg}}\).}
    \label{fig:redtest9&11}
\end{figure}

For both the penalty intensities $\lambda_{\text{reg}}$ and the choices of $\mathcal{P}(f_c)$, all three methods (FPI, NN and FBS) produce largely consistent results. 
Large $\lambda_{\text{reg}}$ shifts $\hat f_c$ closer to $f_c^{\text{initial}} \equiv 1$ but also raises $\E[\log \hat L_T]$, affirming the trade-off between counter-deception and avoiding skepticism. The logarithmic penalty consistently yields results closer to 1, aligning with its steeper decay near zero.

Numerically, $\E[\log \hat L_T]$ under $\hat f_c$ is sensitive to both the penalty and $\lambda_{\text{reg}}$. With the logarithmic penalty, values increase from about $0.50$ at $\lambda_{\text{reg}}=0.1$ to $5.00$ at $\lambda_{\text{reg}}=1.0$, while the quadratic penalty yields smaller increases (from $0.04$ to $2.17$). For reference, the unregularized baseline $f_c^{\text{initial}}\equiv1$ gives $\E[\log \hat L_T]=23.21$, regardless of penalty. These results confirm that by tuning $f_c$, the red team can counteract deception while controlling skepticism, thereby shaping the blue team’s misdirection strategies.


\section{Conclusions and Future Works}\label{sec:discussion}
This paper proposes a deceptive decision-making framework using a partially observable Stackelberg game. The blue team optimizes its controls to mislead the red team, while the red team strategically manipulates the blue team’s misdirection process to expose its true objectives. Using a linear-quadratic setting, we derive a semi-explicit solution and validate the effectiveness of deception-driven strategies through numerical experiments. By integrating ancient strategic principles with modern control and game theory, this work provides novel insights on deception in adversarial systems. Future research will explore extensions to nonlinear dynamics, multi-agent reinforcement learning for adaptive deception and detection, robustness against evolving adversaries, and large-scale interactions in mean-field games with common noise. These directions aim to deepen the theoretical foundations of deception-aware control and expand its applicability to complex networked systems.


\end{document}